\documentclass[11pt]{amsart}
\usepackage{amssymb,amsmath,amsthm,latexsym}
\usepackage[all]{xy}
\usepackage{amssymb,graphicx,color}
\usepackage{marvosym}
\usepackage{amsthm}
\usepackage{amscd}
\usepackage{xspace}
\usepackage{aurical}
\usepackage[T1]{fontenc}
\usepackage{mathrsfs}
\usepackage{enumerate}
\usepackage{mdwlist}
\DeclareMathAlphabet{\mathpzc}{OT1}{pzc}{m}{it}

\newtheorem{theorem}{Theorem}
\newtheorem*{theorem*}{Theorem}
\newtheorem*{lemma*}{Lemma}

\newtheorem{corollary}[theorem]{Corollary}

\newtheorem{definition}[theorem]{Definition}
\newtheorem{example}[theorem]{Example}
\newtheorem{lemma}[theorem]{Lemma}

\newtheorem{observation}[theorem]{Observation}

\newtheorem{proposition}[theorem]{Proposition}
\newtheorem{remark}[theorem]{Remark}

\newcommand\h{{\mathpzc{h}}}

\renewcommand\k{{\mathpzc k}}

\newcommand\hh{{\hat{\h}}}
\newcommand\hk{{\hat{\k}}}

\newcommand\C{{\mathcal C}}
\newcommand\cH{{\mathcal H}}

\newcommand\hH{\hat{\mathcal H}}

\newcommand\Aut{\operatorname{Aut}}

\newcommand\Z{\mathbf{Z}}
\newcommand\R{\mathbf{R}}
\newcommand\ep{{\epsilon}}

\newcommand\abs[1]{\left|#1\right|}

\renewcommand\paragraph[1]{{\bigskip\noindent{\bf #1}.}}

\newcommand{\cat}{{\upshape CAT(0)}\xspace}

\thanks{The second author was supported by The Israel Science Foundation (grant 1026/15)}

\title{Uniform exponential growth for CAT(0) square complexes}
\begin{document}
\author{Aditi Kar and Michah Sageev}

\begin{abstract}
In this paper we start the inquiry into proving uniform exponential growth in the context of groups acting on \cat cube complexes. We address free group actions on \cat square complexes and prove a more general statement. This says that  if $F$ is a finite collection of hyperbolic automorphisms of a \cat square complex $X$, then either there exists a pair of words of length at most 10 in $F$ which freely generate a free semigroup, or all elements of $F$ stabilize a flat (of dimension 1 or 2 in $X$). As a corollary, we obtain a lower bound for the growth constant, $\sqrt[10]{2}$, which is uniform not just for a given group acting freely on a given \cat cube complex, but for all groups which are not virtually abelian and have a free action on a \cat square complex.
\end{abstract}

\maketitle


\section{Introduction}

Given a group $G$ and a finite generating set $S$, we let $\C(G,S)$ denote the Cayley graph of $G$ relative to $S$. The length of an element $g\in G$ with respect to the word metric relative to $S$ is denoted $\abs{g}_S$ and we let $B(S,n)$ denote the ball of radius $n$ in $\C(G,S)$. The \emph{exponential growth rate of $G$ relative to $S$} is defined to be the following limit (which always exists):
\[\omega(G,S) = \lim_{n\to\infty} \abs{B(S,n)}^{1/n}. \]
The \emph{exponential growth rate of $G$} is then given by 

$$\omega(G) = \inf\{\omega(G,S)\vert \text{ finite generating sets } S\}.$$
The group $G$ is said to have exponential growth if $\omega(G,S)>1$ for some and therefore for all finite generating sets $S$. Moreover, $G$ is said to have uniform exponential growth if $\omega(G) > 1$. See de la Harpe \cite{delaHarpe2002} for details.

Gromov asked if every group of exponential growth is also of uniform exponential growth. The first example of a group with exponential growth which is not of uniform exponential growth was constructed by Wilson \cite{Wilson}. Wilson's group and future counterexamples were finitely generated. Whether Gromov's question has an affirmative answer for finitely presented groups remains open. 

Uniform exponential growth is known to hold for groups with virtually free quotients, hyperbolic groups, soluble groups, linear groups in characteristic zero and groups acting on trees in the sense of Bass Serre theory (see \cite{delaHarpe2002} and references therein). Uniform exponential growth is typically established by constructing free semigroups \cite{AlperinNoskov2002}.
	
\begin{lemma*} Let $G$ be a group. Suppose there exists a constant $C>0$ such that for any finite generating set $S$ of $G$, one can find two elements $u,v\in G$ with $\max\{\abs{u}_S, \abs{v}_S\}< C$ and $u$ and $v$ freely generate a free semigroup. Then $\omega(G) \geq \sqrt[C]{2}$
\end{lemma*}

This method and variations of it often allow one to establish ``uniform uniform exponential growth''. Bucher and de la Harpe considered actions on trees and showed in \cite{BHarpe} that the constant in the above lemma is $\sqrt[4]{2}$ for non-degenerate amalgams and HNN extensions. Mangahas \cite{Mangahas} proved that finitely generated subgroups of the mapping class group $Mod(S)$ of a surface $S$ which are not virtually abelian have uniform exponential growth with minimal growth rate bounded below by a constant depending exclusively on the surface $S$. Breuillard \cite[Main Theorem]{EB} established a different sort of uniformity for linear groups: for every $d\in \mathbb{N}$ there is $N(d)\in \mathbb{N}$ such that if $K$ is any field and $F$ a finite symmetric subset of $GL_{d}(K)$ containing $1$, either $F^{N(d)}$ contains two elements which freely generate a nonabelian free group, or the group generated by $F$ is virtually solvable. We refer the reader to \cite{Button} for further examples. 

In this paper we start the inquiry into proving uniform exponential growth in the context of groups acting on \cat cube complexes. We address free group actions on \cat square complexes. We do this by proving a more general statement about groups generated by hyperbolic elements. 

\begin{theorem}
Let $F$ be a finite collection of hyperbolic automorphisms of a \cat square complex. Then either

\begin{enumerate}  
\item there exists a pair of words of length at most 10 in $F$ which freely generate a free semigroup, or 
\item there exists a flat (of dimension 1 or 2) in $X$ stabilized by all elements of $F$.
\end{enumerate}
\label{MainTheorem}
\end{theorem}

As a corollary, we obtain a "uniform uniform" type result, which says that there is a uniform lower bound for growth, not just for a given group, but for all groups acting freely on any \cat square complex.
\begin{corollary}
Let $G$ be a finitely generated group acting freely on a \cat square complex. Then either $w(G)\geq \sqrt[10]{2}$ or $G$ is virtually abelian.
\label{MainCorollary}
\end{corollary}

We expect that a similar result will hold for all dimensions, in that for a finitely generated  group $G$ acting freely on a \cat cube complex of dimension $n$, $G$ will be virtually abelian or $w(G)\geq w_0>1$ where, $w_0$ will depend only on the dimension $n$, and not on the group or the complex. 


We would like to thank the referee for many helpful comments and in particular, for pointing out an error in the original statement  of Theorem 1. 

\section{Hyperplanes and group elements}

We review some relevant basic facts regarding hyperplanes and halfspaces. See, for example, \cite{CapraceSageev2011} or \cite{Sageev2014} for more details.  We let $X$ be a \cat square complex.  We use $\h, \k$ to denote halfspaces, $\hh, \hk$ to denote the corresponding  hyperplanes and $\h^*, \k^*$ to denote the complementary halfspaces. 

We let $\Aut(X)$ denote the collection of cubical, inversion-free automorphisms of $X$. (An inversion is an isometry of $X$ that preserves a hyperplane and inverts the corresponding halfspaces). If $G$ is an action on $X$ which contains inversions, then we may subdivide $X$ so that there are no inversions.

In a \cat cube complex of dimension $n$, any collection of $n+1$ hyperplanes contains a disjoint pair. In particular, in the case of our 2-dimensional complex, if $g\in\Aut(X)$ and $\hh$ is a hyperplane, then the triple $\{\hh,g\hh,g^2\hh\}$ contains a pair that is either disjoint or equal. Thus, either $g^2\hh=\hh$, or one of the pair $\{\hh, g\hh\}$, $\{\hh,g^2\hh\}$ is a disjoint pair. 

Given a hyperplane $\hh$ in $X$ and  $g\in \Aut(X)$ a hyperbolic isometry of $X$, we say that $g$ \emph{skewers} $\hh$ if for some choice of halfspace $\h$ associated to $\hh$, we have  $g^2\h \subset\h$ (note that this includes the case $g\h\subset\h$). This property is equivalent to saying that any axis for $g$ intersects $\hh$ in a single point.

We say that a hyperbolic isometry $g\in\Aut(X)$ is \emph{parallel} to $\hh$ if any axis for $g$ is a bounded distance from $\hh$; and, a hyperbolic isometry is \emph{peripheral} to $\hh$ if it neither skewers $\hh$ nor is parallel to  $\hh$. In this case, any axis lies in a halfspace $\h$ bounded by the hyperplane $\hh$ and is not contained in any neighborhood of $\hh$. It follows that either $g\h^*\subset\h$
 or $g^2\h^*\subset\h$. 

\begin{definition}
Given a hyperbolic isometry $g\in\Aut(X)$, we define the \emph{skewer set of $g$}, denoted $sk(g)$, as the collection of all hyperplanes skewered by $g$.  
We define a \emph{disjoint skewer set} for $g$ as a collection of disjoint hyperplanes in $sk(g)$ which is invariant under $g^2$. 
\end{definition}

If $g$ is parallel to a hyperplane $\hh$, then any hyperplane in $sk(g)$ intersects $\hh$. Since there are no intersecting  triples of  hyperplanes in $X$, this means that no two hyperplanes in $sk(g)$ intersect. Furthermore, any two translates of $\hh$ under $<g>$ are parallel to $g$ and hence cross every hyperplane in $sk(g)$. Again, by the two dimensionality of $X$, this means that the two translates of $\hh$ under $<g>$ are disjoint. We record this observation, since we will make use of it. 

\begin{observation}
If $g$ is parallel to $\hh$, then all the hyperplanes in $sk(g)$ are disjoint and two distinct hyperplanes in the orbit of $\hh$ under $<g>$ are disjoint.
\label{ParallelMeansDisjoint}
\end{observation}

\begin{lemma}
Let $g$ be a hyperbolic automorphism of $X$, then $sk(g)$ is a union of finitely many disjoint skewer sets. 
\end{lemma} 

\begin{proof}
Consider  $\hh\in sk(g)$.  If $g\hh\cap\hh=\emptyset$, we let $P_1=\{g^n(\hh)\vert n\in\Z\}$.
Otherwise, since $X$ is 2-dimensional, we have $g^2\hh\cap\hh=\emptyset$. We then set $P_1=\{g^{2n}(\hh)\vert n\in\Z\}$ and $P_2=\{g^{2n+1}\hh\vert n\in\Z\}$.  Thus $P_1$ and $P_2$ break up the orbit of $\hh$ under $<g>$ into two disjoint skewer sets. Since there are finitely many orbits of hyperplanes in $sk(g)$ under the action of $<g>$, this breaks up $sk(g)$ into finitely many disjoint skewer sets.
\end{proof}

\begin{example}
Let $X$ denote the Euclidean plane, squared in the usual way by unit squares. Let $g$ be an integer translation in the vertical direction. Then the skewer set of $g$ is the collection of horizontal hyperplanes and the number of disjoint skewer sets depends on the translation length of $g$. \end{example}

\begin{example}
Again, let $X$ denote the Euclidean plane. Let $g$ be a glide reflection along the diagonal axis: $g(x,y)=(y+1,x+1)$. Then the skewer set of $g$ is a union of four disjoint skewer sets, each invariant under $g^2$. 
\end{example}
\section{The parallel subset of an element}

Given a hyperbolic $g\in\Aut(X)$, we describe  combinatorially a certain invariant subcomplex associated to $g$ which consists of all the lines parallel to axes in $G$.  (This subcomplex is discussed as well in \cite{HuangJankiewiczPrzytycki2015} and is slightly different than the minimal set of $G$, as described in \cite{Bridson99} or \cite{FernosForesterTao2016}.)

We consider the following partition of hyperplanes $\hH$ of $X$. Let 

\bigskip
$\hH_\parallel (g)=\{\hh\vert \hh \text{ intersects every hyperplane in } sk(g)\}$

\bigskip
$\hH_P(g) =\hH-(sk(g)\cup\hH_\parallel (g))$

Since the elements of $\hH_P(g)$ are peripheral to $g$, it follows that for each hyperplane $\hh\in\hH_P(g)$, there exists a well-defined halfspace $\h$ containing all the axes of $g$. Recall that the collection of cubes intersecting a hyperplane $\hh$ has a product structure $\hh\times [0,1]$. We let $N(\hh)=\hh\times(0,1)$. For a halfspace $\h$ we let $R(\h)=\h- N(\hh)$.

We define 

$$Y_g=\bigcap_{\ell_g\in\h \text{ and }\hh\in\hH_P(g)} R(\h)$$

The subspace $Y_g$ is a $<g>$-invariant convex subcomplex of $X$, and as $Y_g$ contains the axes of $g$, it is non-empty.

The hyperplanes intersecting $Y_g$ are the hyperplanes of  $sk(g)$ and $\hH_\parallel(g)$. Since $sk(g)$ and $\hH_\parallel(g)$ are transverse collections of  hyperplanes, we obtain (by \cite{CapraceSageev2011}) that $Y_g$ admits a product structure $Y_g\cong E_g \times T_g$, where $E_g$ is defined by the hyperplanes $sk(g)$ and $T_g$ is defined by the hyperplanes in $\hH_\parallel(g)$.  Note that  $sk(g)$ does not contain any disjoint facing triples of hyperplanes. As  $g$ does not skewer any hyperplane in $\hH_\parallel(g)$, $g$ fixes a vertex in $T_g$. 
Since $Y_g$ is 2-dimensional, there are two possibilities:

\begin{enumerate}
\item $E_g=\R$ and $T_g$ is isomorphic to a tree.
\item $E_g$ is 2-dimensional and $T_g$ is a point. 
\end{enumerate}

We call $Y_g$ the \emph{parallel set} of $g$ and $E_g$ its \emph{Euclidean factor.}

We need a further understanding of $E_g$ in order to conclude that groups that stabilize it have nice properties. 

\begin{lemma}
Let $E_g$ be the Euclidean factor of $Y_g$. Then either $E_g$ is a Euclidean plane or $E_g$ contains an $Aut(E_g)$-invariant line. 
\end{lemma}

\begin{proof}
See \cite{BrodzkiCampbellGuenterNibloWright03} or \cite{CapraceSageev2011}  for a discussion of ultrafilters, intervals and medians, which are used  in the following argument.
We claim first that $E_g$ is an \emph{interval complex}.  That is, there exist two ultrafilters $\alpha$ and $\beta$ on $\cH$ such that $\overline{E_g}=[\alpha,\beta]$ (where $\overline{E_g}$ denotes the ultrafilter closure of $E_g$). To see this, choose a point on an axis $\ell_g$  for $g$ and let $R^+$ and $R^-$ be the two subrays of $\ell_g$ defined by $p$. Define two ultrafilters
$$\alpha_+= \{\h\in\cH\vert R^+\cap \h \text{ is unbounded} \}$$
$$\alpha_-= \{\h\in\cH\vert R^-\cap \h \text{ is unbounded} \}$$
Note that since $\ell_g$ intersects every hyperplane of $E_g$, $\alpha_+$ and $\alpha_-$ are ultrafilters. Moreover, $\alpha_+$ and $\alpha_-$ make the opposite choices for each hyperplane, which is to say $\alpha_+\cap\alpha_-=\emptyset$. It follows that for every other ultrafilter $\beta$, we have that 

$$med(\alpha_+,\alpha_-,\beta)= (\alpha_+\cap\alpha_-)\cup(\alpha_+\cap\beta)\cup(\alpha_-\cap\beta)=\beta$$
This means that $\overline{E}_g=[\alpha_+,\alpha_-]$, as claimed.

It follows, by \cite{BrodzkiCampbellGuenterNibloWright03}, Theorem 1.16,  that $E_g$ embeds isometrically in the standard squaring of the Euclidean plane. We can thus assume that $E_g$ is an isometrically embedded subset of the standard squaring of the Euclidean plane. It follows that the hyperplanes in $E_g$ are either lines, rays or closed intervals. Since $g\in Aut(E_g)$ is a hyperbolic element, we also have that there are finitely many orbits of hyperplanes under the action of $Aut(E_g)$ on $E_g$. 

If all the hyperplanes are lines, then we obtain that $E_g$ is itself a Euclidean plane and we are done. If some hyperplane, say a horizontal one, is a ray, then we claim that all the other horizontal hyperplanes are rays. For if some horizontal hyperplane is a line, then by the fact that $g$ is acting cofinitely on the hyperplanes, we would obtain two horizontal line hyperplanes, separated by a horizontal ray hyperplane. This would contradict the fact that $E_g$ is isometrically embedded in the Euclidean plane. By the same reasoning, there can be no closed interval horizontal hyperplanes, for we would obtain two ray intervals a bounded Hausdorff distance apart in $E_g$ separated by a closed interval hyperplane. From this it follows that all the vertical hyperplanes are rays as well and we have that $E_g$ is a  "staircase", as in Figure \ref{Staircase}. 

\begin{figure}[h]
\includegraphics{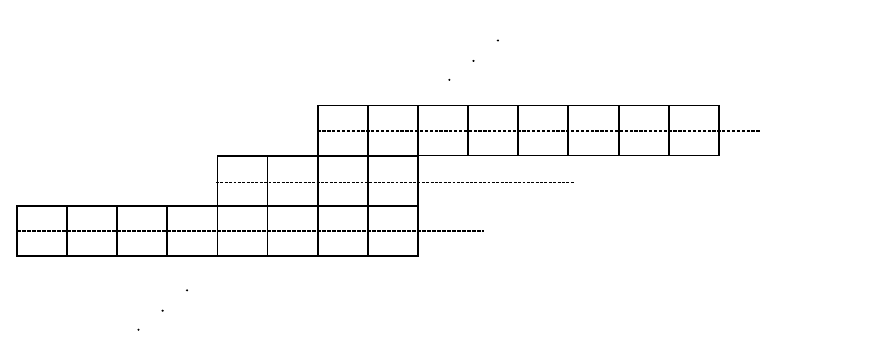}
\caption{The case in which all hyperplanes in $E_g$ are rays. The endpoints of the rays are invariant, and hence any line in $E_g$ a bounded distance from all endpoints is 
$\Aut(E_g)$-invariant}
\label{Staircase}
\end{figure}

In this "stairstep" case, the space of lines which coarsely contains the endpoints of the hyperplanes is itself a ray $R$ which is $\Aut(E_g)$-invariant, hence there is an $\Aut(E_g)$ fixed point in $R$ and hence an $\Aut(E_g)$-invariant line in $E_g$.

 If there exists a hyperplane in $E_g$ which is a closed interval, then by similar considerations as above, we may conclude that all hyperplanes are closed intervals. Since $<g>$ acts cocompactly on $E_g$, it follows all lines in $E_g$ are parallel and the space of such lines is a compact interval $I$. Since the action of $\Aut(E_g)$ on $I$ has a fixed point, it then follows that there is an $\Aut(E_g)$-invariant line. 
\end{proof}

\section{The ping-pong lemma and hyperplane patterns that yield free semigroups}

We will use the following version of the Ping Pong Lemma (see, for example, \cite{deLaHarpe2000})

\begin{lemma}[Semigroup Ping Pong] Suppose that a group $G$ is acting on a set $X$ and $U, V$ are \emph{disjoint} subsets of $X$. The elements $a, b \in G \backslash \{1\}$ satisfy  
\begin{itemize} 
\item $a(U \cup V)\subset U$
\item $b(U \cup V) \subset V$ 
\end{itemize}

\noindent Then $a$ and $b$ freely generate a free subsemigroup in $G$. 
\label{SemigroupPingPong}
\end{lemma} 

\begin{proof}
Let $\Sigma$ be the semigroup generated by $a$ and $b$ in $G$. Observe that for any $g,h \in \Sigma \subset G$, $ag=ah$ or $bg=bh$ in $\Sigma$ if and only if $g=h$ in $\Sigma$. Therefore, it is enough to check that two words of the form $ag$ and $bh$ cannot be equal in $\Sigma$. But, $ag(U \cup V) \subset U$ and $bh(U\cup V) \subset V$. Since, $U \cap V = \emptyset$, $ag \neq bh$. 
\end{proof}

\subsection{On groups acting on trees}

To warm up, and to record a few observations we use later on, we first explore what happens for a pair of hyperbolic isometries acting on a tree. We include the proofs here because we will need these type of arguments. However, this is not new. See, for example, \cite{AlperinNoskov2002}. Let $T$ be a simplicial tree. Recall if an element $g$ of $Aut(T)$ is hyperbolic then there is a unique geodesic  $\ell_g$ (called the axis of $g$) which is invariant under $g$ on which $g$ induces a translation.

\begin{proposition}
If $a$ and $b$ are two hyperbolic automorphisms of a tree $T$, then one of the following occurs: 
\begin{itemize}
\item $a, b$ share the same axis,
\item $a^{\pm 1}$ and $b^{\pm 1}$ freely generate a free semigroup. 
\end{itemize}
\label{TreeCase}
\end{proposition}

\begin{proof}  
Suppose that $\ell_a\not=\ell_b$. First assume that $\ell_a\cap \ell_b$ is non-empty and contains an edge $e=[p,q]$. (See Figure \ref{Tree}.) Choose $e$ so that $q$ is a point of bifurcation of $\ell_a$ and $\ell_b$. Let $T_q$ be the component of $T-interior(e)$ containing $q$. After possibly replacing $a$ by $a^{-1}$ and/or $b$ by $b^{-1}$, we see that $ae \subset T_q$ and $be\subset T_q$. Set $U=aT_q$ and $V=bT_q$. Then $U, V$ satisfy the hypothesis of Lemma \ref{SemigroupPingPong}. We will generalize this argument in our context.  


\begin{figure}[h]
\includegraphics{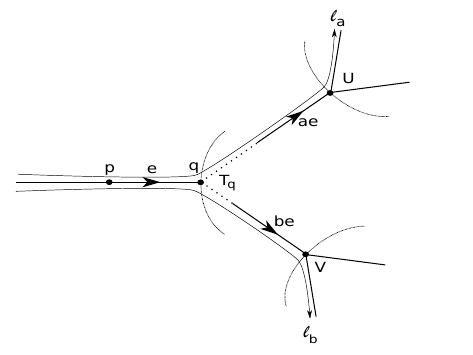}
\caption{The hyperbolic isometries $a$ and $b$ have non-equal, but overlapping axes.}
\label{Tree}
\end{figure}

\begin{figure}[h]
\includegraphics{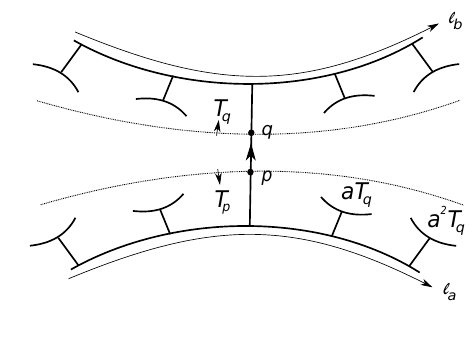}
\caption{The hyperbolic isometries $a$ and $b$ have disjoint axes.}
\label{Tree2}
\end{figure}

The case when $\ell_a\cap\ell_b=\emptyset$ calls for a different argument (See Figure \ref{Tree2}). Consider an edge $e=[p,q]$ situated along the geodesic arc joining $\ell_a$ and $\ell_b$. Let $T_p$ be the component of $T-\textrm{interior}(e)$ containing $p$ and $T_q$ be the component of $T-\textrm{interior}(e)$ containing $q$. Suppose (without loss of generality) that $\ell_a\subset T_p$ and $\ell_b\subset T_q$.  Then letting $U=\bigcup_{n>0} a^nT_q$ and $V=\bigcup_{n>0} b^n T_p$, we see that $a(U\cup V)\subset U$ and $b(U\cup V)\subset V$, as required. In fact, in this case, we can argue that $a$ and $b$ generate a free group by adjusting $U$ and $V$ to include all non-zero powers of $a$ and $b$, but we will not need this fact. 
Note that there is a singular case in which $\ell_a$ and $\ell_b$ intersect in a single point. In this case, we simply use the intersecting vertex to separate $T$ into two subtrees, each containing a different axis, and proceed in the same manner. 
\end{proof}

\subsection{Back to \cat cube complexes}

The following Lemma works in any dimension and so, just for the paragraph below, we let $X$ be an $n$-dimensional \cat cube complex.

\begin{lemma}
Let $g_1,g_2\in \Aut(X)$ and suppose that there exists a halfspace $\h$ of $X$ such that $g_i\h\subset \h$ and $g_1\h \subset g_2\h^*$.  Then $g_1,g_2$ generate a free semigroup.
\label{FreeSG1}
\end{lemma}

\begin{proof} 
This argument resembles the first case in the proof of Proposition \ref{TreeCase}. Set $U=g_1\h$ and set $V=g_2\h$ and apply Lemma \ref{SemigroupPingPong}
\end{proof}

\noindent \emph{We call the triple $\{\h, g_1\h, g_2\h\}$ a \emph{ping pong triple} for $g_1$ and $g_2$.}

\section{Main argument}

Now, let $X$ be a \cat square  complex.

\begin{lemma}[All or nothing]
Let $a$ and $b$ be hyperbolic isometries of $X$ and let $P$ be a disjoint skewer set for $a$. Suppose that no pair of words of length at most $6$ in $a$ and $b$   generate a free semigroup, then 
either $b$ skewers every hyperplane in $P$ or $b$ does not skewer any hyperplane in $P$. 
\label{AllOrNothing}
\end{lemma}

\begin{proof} 
Recall that for any $\hh$ in $sk(a)$, there exists an associated halfspace $\h$ such that $a^2\h\subset\h$. If $b$ skewers some element in $P$, but not all, we may also choose $\h$ such that  $\h$  is skewered by $b$ but $a^2\h$ is not skewered by $b$. After replacing $b$ possibly by $b^{-1}$, we may assume  that $b^2\h\subset h$. Note that $b$ and hence $b^2$  is peripheral to $a^2\hh$. 

Now by the 2-dimensionality of $X$, either $b^2 a^2 \hh\cap a^2\hh=\emptyset$ or $b^4a^2\hh\cap a^2\hh=~\emptyset$. We further have that $b^2a^2\h\subset b^2\h\subset \h$ and $b^4a^2\h\subset b^4\h\subset\h$. 

We thus have that either $\{\h, a^2\h,b^2a^2\h\}$ or $\{\h,a^2\h,b^4a^2\h\}$ is a ping pong triple of halfspaces for the pairs $a^2,b^2a^2$ or $a^2,b^4a^2$. See Figure \ref{allnone}. In either case, we obtain words of length at most 6 freely generating a free semigroup, a contradiction.

\begin{figure}[h]
\includegraphics{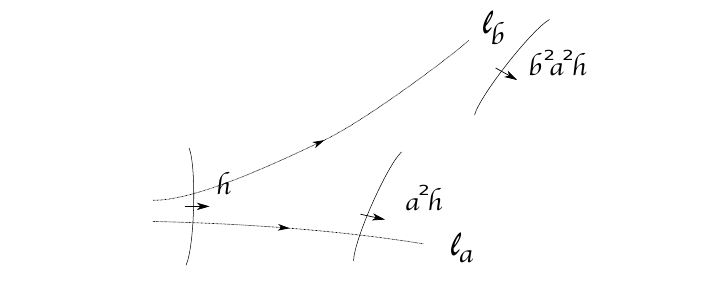}
\caption{The element $b$ skewering $\h$ but not $a\h$. }
\end{figure}
\label{allnone}
\end{proof}

\begin{proposition}[Not skewering means parallel]
Let $a$ and $b$ be hyperbolic isometries of $X$ and let $P$ be a disjoint skewer set for $a$. Let $\ell_b$ be an axis for $b$. Suppose that $b$ does not skewer any element of $P$ and that no pair of words of length no more than $10$ freely generate a free semigroup. Then
\begin{enumerate}
\item the  axis $\ell_b$ is parallel to every hyperplane $\hh\in P$. 
\item $bP\in sk(a)$
\item $b^2$ stabilizes every hyperplane in $P$. 
\end{enumerate}
\label{MainProp}
\end{proposition}

\begin{proof} The disjoint skewer set $P$ decomposes as a finite union of $\langle a^2 \rangle$-orbits. So, the assumption that $b$ does not skewer any hyperplane in $P$ holds for each orbit. If the conclusion of the Proposition holds for each $\langle a^2 \rangle$-orbit, then it holds for all of $P$. Therefore it suffices to prove the Proposition for when $P$ is a single $\langle a^2 \rangle$-orbit: there exists $\h\in P$ such that $a^2\h\subset\h$ and  $P=\{a^{2k}\hh\vert k\in\Z\}$. We set $c=a^2$. Since $b$ does not skewer any hyperplane in $P$, we may assume that $\ell_b\subset \h\cap c\h^*$. (We are using here that the action is without inversions, so that if $\ell_b\subset \hh$ for some hyperplane, there is a parallel axis for $b$ on either side of $\hh$.) We will now use our assumptions to remove the possibility that $b$ is peripheral to $\hh$ or $c\hh$. 

First, suppose $b$ is peripheral to both $\hh$ and $c\hh$. We claim that we can  find a facing triple of hyperplanes of the form $\hh,b^s\hh,b^t\hh$ with $|s|,|t|\leq 4$. 

To see this, consider the 6 translates $\{b^{-2}\hh,b^{-1}\hh,\hh,b\hh,b^2\hh, b^3\hh \}$. Construct the intersection graph $\Gamma$ for these six hyperplanes: the vertices of $\Gamma$ are the elements of $\{b^{-2}\hh,b^{-1}\hh,\hh,b\hh,b^2\hh, b^3\hh \}$, and two vertices are joined by an edge if and only if the respective hyperplanes cross. Since $R(3,3)=6$,   the graph $\Gamma$ possesses a clique or an anti-clique on 3 vertices. However, as in a CAT(0) square complex, three distinct hyperplanes cannot pairwise intersect, the intersection graph $\Gamma$ must have an anti-clique  $T$ consisting of three hyperplanes. If $T$ contains $\hh$, then we are done; else, we take a suitable translate of $T$. The highest exponents appear when $T=\{b^{-2}\hh,b^2\hh, b^3\hh \}$ and in this case, we take $b^{-2}T$ as our chosen set of facing triples. 



We now have $s,t$ of absolute value at most 4, such that $\hh$, $b^s\hh$ and $b^t\hh$ are disjoint and form a facing triple. Translating by $c$, we get that $c\hh$, $cb^s\hh$ and $cb^t\hh$ form a facing triple of hyperplanes. As $b$ is also peripheral to $c\hh$, there exists $\eta \leq 2$ such that $b^\eta c \hh \cap c\hh = \emptyset$. Now, $cb^s\h^*$ and $cb^t\h^*$ are both disjoint half-spaces that lie inside the half-space $b^\eta c\h^*$. This implies that the two elements $cb^sc^{-1}b^{-\eta}$ and $cb^tc^{-1}b^{-\eta}$ (each of length $\leq 10$) freely generate a free semigroup, a contradiction. 

Let us now assume that $b$ is parallel to $\hh$ but peripheral to $c\hh$. It follows from Observation \ref{ParallelMeansDisjoint} that for any $i\in\Z$, $b^i\hh=\hh$ or $b^i\hh\cap\hh=\emptyset$. First let us consider the case that $b^2\hh=\hh$. Note that since we are assuming that $\Aut(X)$ acts with no inversions, we have that $b^2\h=\h$. Now since $b$ is peripheral to $c\hh$, for $k=1$ or $2$, we have that $b^{2k} c\hh\cap c\hh=\emptyset$. We thus obtain a ping pong triple of halfspaces $\{\h, c\h, b^{2k} c\h\}$ for the elements $c$ and $b^{2k}c$. From Lemma \ref{FreeSG1} we see that $c$ and $b^{2k }c$ freely generate a free semigroup, a contradiction since these are words of length at most $6$ in $a$ and $b$. (See Figure \ref{ParallelPeripheral1}.)

\begin{figure}[ht]
\includegraphics{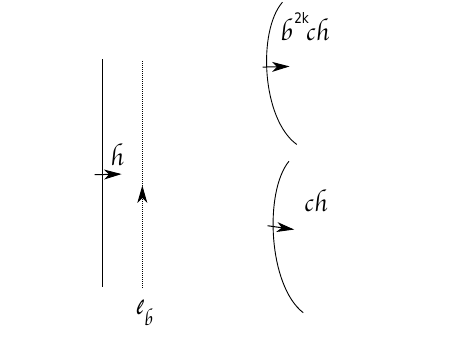}
\caption{If $b$ stabilizes $h$, we obtain a ping-pong triple of hyperplanes. }
\label{ParallelPeripheral1}
\end{figure}

We may thus assume that $b\hh\cap\hh=\emptyset$ and $b^2\hh\cap\hh=\emptyset$. Only one of $b\hh, b^2\hh$ can separate $\hh$ and $c\hh$, for otherwise we would have $b\h\subset b^2\h$ or $b^2\h\subset\h$. So for some $\ep=1$ or $2$, we can assume that $b^\ep\hh$ does not separate $\hh$ and $c\hh$. Note also that since $c\hh$ is peripheral to $b$, one cannot have $b^\epsilon\hh\subset c\h$.

If $c\hh\cap b^\ep\hh=\emptyset$, then we obtain a ping-pong triple of halfspace $\{c\h^*,\h^*,b^\ep\h^*\}$ for the words $c^{-1}$ and $b^\ep c^{-1}$. Since these are words of length at most $4$ in $a$ and $b$, we have a contradiction. (See Figure \ref{ParallelPeripheral2}.)

\begin{figure}[ht]
\includegraphics{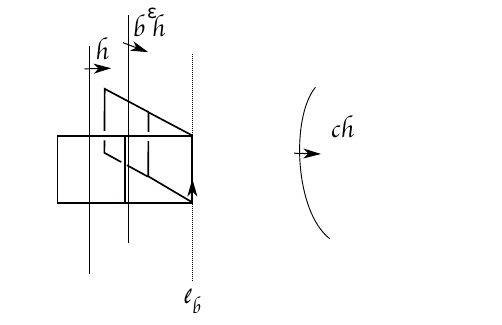}
\caption{If $c\hh\cap b^\ep\hh=\emptyset$ and $b^\ep\hh$ does not separate $\hh$ and $c\hh$, we obtain a ping pong triple }
\label{ParallelPeripheral2}
\end{figure}

Thus we assume that $b^\ep\hh\cap c\hh\not=\emptyset$ and refer to Figure \ref{ParallelPeripheral3}. Since, by Observation \ref{ParallelMeansDisjoint}, any hyperplane in $sk(b)$ intersects $b^\ep \hh$, and we are assuming that $b^\ep \hh\cap c\hh\not=\emptyset$, the 2-dimensionality of $X$ implies that any hyperplane in $sk(b)$ is disjoint from $c\hh$. Moreover, by Observation \ref{ParallelMeansDisjoint}, we have that for any hyperplane $\hk$ in $sk(b)$, 
$b\k\subset k$ for some choice of halfspace $\k$ associated to $\hk$. We may further choose $\k$ such that $c\h\subset \k\cap b\k^*$.

Applying $c^{-1}$, we see that $\h\subset c^{-1}\k\cap c^{-1}b\k^*$. Applying $b^\ep$, we now see that $b^\ep c^{-1}\hk\subset b^\ep\h^*\subset\h$. Thus we have a ping pong triple of half spaces $\{c^{-1}b\k^*, c^{-1}\k^*,  b^\ep c^{-1}\k^*\}$ for the elements  $c^{-1}b^{-1}c$ and $b^\ep c^{-1} b^{-1} c$. So by Lemma \ref{SemigroupPingPong} we have that $c^{-1}b^{-1}c$ and $b^\ep c^{-1} b^{-1} c$ generate a free semigroup and these are words of length at most $7$.

\begin{figure}[h]
\includegraphics{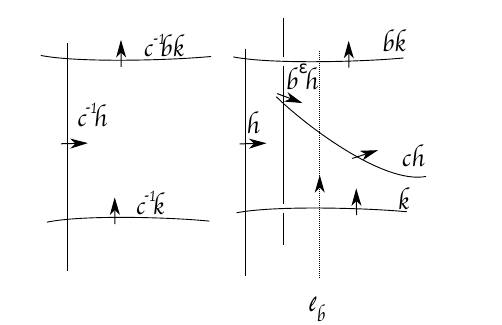}
\caption{If $c\hh\cap b^\ep\hh\not=\emptyset$, we obtain a ping pong triple.}  
\label{ParallelPeripheral3}
\end{figure}

We may thus assume that $b$ is parallel to both $\hh$ and $c\hh$. Assume,  that $d(\ell_b, \hh)\leq d(\ell_b, c\hh)$. (There is no loss of generality here, for if $d(\ell_b, c\hh)\leq d(\ell_b, \hh)$, we will reverse the roles of $\hh$ and $c\hh$ in the following argument.) 

 As before, we first consider what happens if $\hh$ is not stabilized by $b^2$. Here we obtain $\hh, b\hh$ and $b^2\hh$ are disjoint. 
 We cannot have that $b\hh=c\hh$ or $b^2\hh=c\hh$, for then we would obtain $c^{-1}b\hh$ or $c^{-1}b^2\hh$ is an inversion of $\hh$. 
 Thus, we have that $b\hh\subset c\h^*$ and $b^2\hh\subset c\h^*$. We now proceed as in the case in which $c\hh$ is peripheral to $b$ to produce a ping pong triple of halfspaces $\{c\h^*,\h^*,b^\ep\h^*\}$ for the words $c^{-1}$ and $b^\ep c^{-1}$. (The configuration is the same as in Figure \ref{ParallelPeripheral2} except that here $c\hh$ is parallel to $\ell_b$.)
 
So assume $b^2\hh=\hh$. Again, as above, if $b^2$ did not also stabilize $c\hh$, we would obtain a small ping pong triple. Thus $b^2$ stabilizes $c\hh$ as well. Since $b^2$ stabilizes $c\hh$ (and the action is inversion-free), we have an axis for $b^2$ in $c\h\cap c^2\h^*$. We can now carry out all the above arguments replacing $\hh$ and $c\hh$ with $c\hh$ and $c^2\hh$ to conclude that $b^2$ stabilizes $c^2\hh$. Proceeding in this way we see that $b$ is parallel to every hyperplane of $P$ and that $b^2 P=P$. 

We are left to show that $bP\subset sk(a)$. We now argue as in the proof of Lemma \ref{AllOrNothing} using the pair $bab^{-1}$ and $a$. The pairs $ba^2b^{-1}, a^2ba^2b^{-1}$ and $ba^2b^{-1}, a^4ba^2b^{-1}$ made of words of length at most 8 in $a,b$ may freely generate free semigroups. But, we have assumed that there are no such free semi-groups. Hence, in our current case, Lemma \ref{AllOrNothing} implies that $a$ skewers every hyperplane in $bP$ or none of the hyperplanes in $bP$. In the former case, we get $bP\subset sk(a)$ as required. So suppose that $a$ does not skewer any hyperplane in $bP$. Note that $b\hh$ must be disjoint from $\hh$ and $c\hh$ because $\ell_b$ is parallel to all three. Similarly $bc\hh$ is disjoint from $\hh$ and $c\hh$. Since $\ell_ b\subset \h\cap c\h^*$, we have either $b\h^*\subset\h\cap c\h^*$ or $bc\h\subset\h\cap c\h^*$, depending on which of $\hh$ or $c\hh$ is closer to $\ell_b$. In either case, we then get a small ping pong triple, a contradicition. 
\end{proof}

If $a$ and $b$ are elements such that there exists a disjoint skewer set $P$ for $a$  as in Proposition \ref{MainProp}, then we say that \emph{$b$ is subparallel to $a$}. 
\begin{corollary}
Given hyperbolic isometries $a$ and $b$ such that no words of length at most 10 generate a free semigroup of rank 2, $b$ is subparallel to $a$ if and only if $sk(a)-sk(b)\not=\emptyset$.
\end{corollary}
\begin{proof}
If $b$ is subparallel to $a$, then by definition, there exists a disjoint skewer set for $a$ such that $b$ is parallel to all the hyperlanes in $P$. Thus $P\subset sk(a)-sk(b)$.
Conversely, if there exists $\hh\in sk(a)-sk(b)$, then by Lemma \ref{AllOrNothing}, the entire  disjoint parallel set $P$ for $a$ containing $\hh$ is not skewered by $b$. 
Then by Proposition \ref{MainProp}, $b$ is subparallel to $a$. 
\end{proof}

From this corollary, we see that there are three possibilities for two hyperbolic elements $a$ and $b$ so that words of length at most 10 do not freely generate a free semigroup.

\begin{enumerate}[I]
\item $sk(a)=sk(b)$
\item $b$ is subparallel to $a$ and $a$ is subparallel to $b$
\item $b$ is subparallel to $a$ and $a$ is not subparallel to $b$ (or the same with the roles of $a$ and $b$ reversed)
\end{enumerate}

We claim that in each of these cases, we can find an invariant line or flat for $<a,b>$. 

\begin{proposition}
Let $a$ and $b$ be hyperbolic isometries such that no words in $a$ and $b$ of length at most 10 freely generate a free semigroup, then there exists a Euclidean subcomplex of $X$ invariant under $<a,b>$.
\label{FlatForAPair}
\end{proposition}

\begin{proof}
We analyze the three cases above. Suppose we are in Case I, so that $sk(a)=sk(b)$. Then we consider $Y=Y_a=Y_b=E \times T$. If $T$ is trivial (i.e a single point), then we have that both $a$ and $b$ leave $E$ invariant, as required. Otherwise we have that $Y= \R\times T$, where $a$ and $b$ both act by vertical translation. We consider the action of $a$ and $b$ on $T$. Both $a$ and $b$ have nonempty fixed point sets, which we denote $F_a$ and $F_b$. If $F_a\cap F_b\not=\emptyset$, then choosing $p\in F_a\cap F_b$ we have that both $a$ and $b$ stabilize the line $\R\times \{p\}\subset \R\times T$. 

So suppose that $F_a\cap F_b=\emptyset$.  As in \cite{Serre77}, we have that $ab$ is hyperbolic in its action on $T$, stabilizing a line $\ell$ which intersects both $F_a$ and $F_b$. We claim that $a$  stabilizes $\ell$. For otherwise, consider the line $a\ell$. This is stabilized by the element $u=a(ab)a^{-1}$. If $a\ell\not=\ell$,  then we obtain that $(ab)^{\pm 1}$ and $u^{\pm 1}$ freely generate a free semigroup by Proposition \ref{TreeCase} , contradicting our assumption. Similarly, we see that $b$ stabilizes $\ell$ as well. Thus $<a,b>$ stabilizes the flat $\R\times \ell\subset \R\times T$, as required. 

We now consider Case II, so that $a$ and $b$ are subparallel to one another. Note that since an axis for $a$ is parallel to a hyperplane (in $sk(b)$), then all the hyperplanes in $sk(a)$ are disjoint.  Similarly all the hyperplanes in $sk(b)$ are disjoint. Note also every hyperplane in $sk(a)$ crosses every hyperplane in $sk(b)$, so that they determine a flat $E=Y_a\cap Y_b$.
Moreover since $b$ is parallel to one of the hyperplanes in $sk(a)$, it is  parallel or peripheral to all hyperplanes in $sk(a)$. But then, Proposition \ref{MainProp} implies that for all disjoint skewer sets $P\subset sk(a)$, we have $bP\subset sk(a)$. Thus $b sk(a)\subset sk(a)$. By the same argument, we obtain $b ^{-1} sk(a)\subset sk(a)$, so that $b(sk(a))=sk(a)$. 

Similarly, we have that $a(sk(b))=sk(b)$.  We thus have that $<a,b>$ stabilizes the flat $E$. 

 Finally, we consider Case III. In this case there exists a disjoint skewer set $P$ for $a$, so that $b$ is parallel to $P$. However, since $a$ is not subparallel to $b$, $a$ also skewers every element in $sk(b)$. Since the hyperplanes in $sk(b)$ all intersect the hyperplanes in $P$, we have that $sk(a)$ has crossing hyperplanes. It follows that the parallel set 
 $Y_a$ for $a$ is of the form $Y_a=E \times \{\text{point}\}$. It is also easy to see that $b$ stabilizes $E$, so that $<a,b>$ stabilizes $E$. 
\end{proof}

We are now ready to prove Theorem \ref{MainTheorem}, which we restate here for convenience.

\begin{theorem*} 
Let $F$ be a finite collection of hyperbolic automorphisms of a \cat square complex. Then either

\begin{enumerate}  
\item there exists a pair of words of length at most 10 in $F$ which freely generate a free semigroup, or 
\item there exists a flat (of dimension 1 or 2) in $X$ stabilized by all elements of $F$.
\end{enumerate}
\end{theorem*}


\begin{proof}
Consider $F=\{s_1,s_2, \ldots,s_n\}$. Each of the pairs $\{s_i,s_j\}$ satisfy one of the cases I, II, or III, above. 

If there exists a pair of type III, without loss of generality, assume that is the pair $\{s_1,s_2\}$, with $s_2$ subparallel to $s_1$ and $s_1$ not subparallel to $s_2$. In this case,
 the parallel set $Y_{s_1}=E\times \{point\}$. In this case, for every other $s_i$, we have that the pair $s_1, s_i$ is either of type I or III. In either case, we obtain that $s_i $ stabilizes $E$ and we are done. 

So we suppose that no pair $\{s_i,s_j\}$ is of type III. Suppose, that there exists a pair, say $\{s_1, s_2\}$, which is of type II. Let $E$ be the flat in $X$ on which $<s_1, s_2>$ acts. For any other $s_i$, we have that the pairs $\{s_1, s_i\}$ and $\{s_2, s_i\}$ is of type I or II. It cannot be that both pairs are of type I since $sk(s_1)\cap sk(s_2)=\emptyset$. Also, it cannot be that $s_i$ is subparallel to both $s_1$ and $s_2$, for otherwise $\ell_{s_i}$ would be parallel to hyperplanes in $sk(s_1)$ and in $sk(s_2)$, but every hyperplane in $sk(s_1)$ crosses every hyperplane in $sk(s_2)$ in a single point. Thus a line cannot be parallel to a hyperplane in $sk(s_1)$ and a hyperplane in $sk(s_2)$. It follows that, without loss of generality, $s_i$ is subparallel to $s_1$ and $sk(s_i)=sk(s_2)$. It then follows that $s_i$ stabilizes $E$. 

Finally, suppose that all the pairs $s_i, s_j$ are of type I. Thus $sk(s_i)=sk(s_j)$ for all $i,j$. Thus $G$ stabilizes  $Y=E\times T= E_{s_i}\times T_{s_i}$. If $E$ contains squares, then $T$ is trivial and $s_i$ stabilizes $E$ as required. So suppose that $Y=\R\times T$, and each $s_i$ acts ``vertically''. That is, $s_i$ acts by translation along $\R$ and has a fixed point in $T$. 

We now examine the action of $G$ on $T$. Let $F_i$ denote the fixed set of $s_i$. If for each pair $i, j$, $F_i\cap F_j\not=\emptyset$, then by a standard result, 
$X_n=\cap_{i=1}^n F_i\not=\emptyset$.  Choose a vertex $p_n\in X_n$. Then $H_n=<s_1,\ldots,s_n>$ acts on $\ell_n=\R\times p_n$ by translations. Thus $H_n$ stabilizes a flat in $X$. 

So suppose that there exists a pair, say $F_1$ and $F_2$, such that $F_1\cap F_2=\emptyset$. In this case, as in the proof of Proposition \ref{FlatForAPair}, there exist a line $\ell\subset T$ on which $<s_1,s_2>$ acts as a dihedral group. As in the proof of Proposition \ref{FlatForAPair}, we also obtain that for every $i$, $s_i$ stabilizes $\ell$. Thus $G$ stabilizes $\ell$, and therefore the flat $\R\times \ell$, as required. 
\end{proof}

\begin{remark}
\emph{The proof of the Theorem shows that in case (1), there is a subset $F_0$ of $F$ made of 2 or 3 elements and a pair of words of length $\leq 10$ in $F_0$ which generate the free semi-group of rank 2.  }
\end{remark}

Corollary \ref{MainCorollary} now follows from the Main Theorem since when the action of a group is free, stabilizing a flat implies the group is virtually abelian, by the Bieberbach Theorem.
 
\end{document}